\documentclass[11pt, a4paper]{article}

\usepackage[margin=1in]{geometry}

\usepackage{amsmath, amssymb, amsthm, mathtools}
\usepackage{mathrsfs}
\usepackage{bm}

\usepackage{enumitem}
\usepackage[colorlinks=true, citecolor=blue, linkcolor=blue]{hyperref}

\newtheorem{theorem}{Theorem}[section]
\newtheorem{corollary}[theorem]{Corollary}

\theoremstyle{remark}

\DeclareMathOperator{\Li}{Li}

\title{Bell Transforms of Arithmetic Functions: Euler Products, Congruences, and Polynomial Sequences}
\author{Mahipal Gurram}
\date{\today}

\begin{document}

\maketitle

\begin{abstract}
We present a unified algebraic framework utilizing the formal Bell transform to bridge the Dirichlet convolution of arithmetic functions with the combinatorial structure of infinite Euler-type products. By analyzing the logarithmic derivative of exponential generating functions, we establish explicit mappings between Bell exponents and Möbius inversions. We apply this framework to derive exact vanishing properties and congruence inheritances for classical sequences, including Ramanujan's tau function and prime-colored partitions. Furthermore, we demonstrate that the inverse Bell transform seamlessly recovers classical partition recurrences and provides a discrete combinatorial engine for generating special polynomial families, including classical Appell and Sheffer sequences.
\end{abstract}
\section{Introduction}

The interplay between the discrete realm of arithmetic functions and the continuous domain of formal power series is a cornerstone of analytic number theory and enumerative combinatorics. The expansion of infinite products into power series and the reciprocal extraction of arithmetic sequences from product exponents underpins the study of integer partitions, modular forms, and special functions \cite{Andrews1976}.

Classical identities, particularly those arising in $q$-series, Euler products, and Ramanujan’s theta functions, rely on deep multiplicative structures encoded through divisor sums and Dirichlet convolutions \cite{Andrews1976,HardyWright2008}. In this paper, we present a unified algebraic framework utilizing the formal Bell transform to systematically bridge the Dirichlet convolution algebra of arithmetic functions with the combinatorial geometry of infinite Euler-type products. By analyzing logarithmic derivatives of exponential generating functions, we establish explicit mappings between Bell exponents and Möbius inversions.

Bell polynomials naturally encode the combinatorics of set partitions and exponential generating functions \cite{Comtet1974,Riordan1968}. Using these structures, we derive exact vanishing properties and congruence inheritances for classical arithmetic sequences, including Ramanujan’s tau function and prime-colored partition functions. We further demonstrate that the inverse Bell transform recovers classical partition recurrences and provides a combinatorial engine for generating Appell and Sheffer polynomial sequences 

\subsection{Notations and Preliminaries}

To establish a rigorous foundation, we define the core notations and algebraic structures utilized throughout this paper.

\paragraph{Sets and Arithmetic Functions}
Let $\mathbb{N}$ denote the set of positive integers, and $\mathbb{C}$ the field of complex numbers. An \emph{arithmetic function} is defined as any map
\[
g:\mathbb{N}\to\mathbb{C}.
\]

\paragraph{Dirichlet Convolution}
The set of arithmetic functions forms a commutative ring under pointwise addition and Dirichlet convolution \cite{Apostol1976}. For two arithmetic functions $f$ and $g$, their Dirichlet convolution $(f*g)$ is defined by
\[
(f*g)(n)
=
\sum_{d\mid n}
f(d)\,
g\!\left(\frac{n}{d}\right).
\]

The identity element for this operation is the Dirichlet identity function
\[
\varepsilon(n)
=
\begin{cases}
1, & n=1,\\
0, & n>1.
\end{cases}
\]

\paragraph{Möbius Inversion}
The Möbius function $\mu(n)$ is the Dirichlet inverse of the constant function $1$ and plays a central role in multiplicative number theory \cite{Apostol1976,HardyWright2008}. If
\[
f(n)=\sum_{d\mid n} g(d),
\]
then
\[
g(n)=\sum_{d\mid n}\mu(d)\,f\!\left(\frac{n}{d}\right),
\]
or equivalently,
\[
g=\mu * f.
\]

\paragraph{Formal Power Series and Euler Products}
We work within the ring of formal power series $\mathbb{C}[[x]]$ \cite{Stanley1999}. Euler-type infinite products of the form
\[
\prod_{m\ge1}(1-x^m)^{\beta(m)}
\]
occur naturally in partition theory, modular forms, and generating-function identities \cite{Andrews1976}.

\paragraph{Complete Exponential Bell Polynomials}
The complete Bell polynomials
\[
B_n(x_1,\dots,x_n)
\]
encode the combinatorial structure of set partitions and exponential compositions \cite{Comtet1974}. They are defined through the exponential generating function
\[
\exp\!\left(
\sum_{m\ge1}x_m\frac{t^m}{m!}
\right)
=
\sum_{n\ge0}
\frac{
B_n(x_1,\dots,x_n)
}{n!}\,
t^n.
\]

These polynomials act as algebraic translators connecting logarithmic derivatives of generating functions with coefficient extractions in formal power series.
\section{Main Results}

We begin by establishing the foundational relationship between exponential generating functions and infinite Euler products via Dirichlet convolution.

\begin{theorem}[General Euler-product form of the Bell transform] \label{thm:euler_product}
Let \(g:\mathbb N\to\mathbb C\) be an arithmetic function, and define the formal Bell transform
\[
F_g(x)
=
\exp\!\left(
-\sum_{n\ge1}\frac{g(n)}{n}x^n
\right).
\]
Then \(F_g(x)\) admits the Euler-product expansion
\[
F_g(x)
=
\prod_{m\ge1}(1-x^m)^{\beta_g(m)},
\]
where the Bell exponents \(\beta_g(m)\) are given by
\[
\beta_g(m)
=
\frac1m
\sum_{d\mid m}\mu(d)\,g(m/d).
\]
Equivalently, this may be written in terms of Dirichlet convolution as \( m\beta_g(m) = (\mu*g)(m) \).
\end{theorem}

\begin{proof}
Taking the formal logarithm of the Euler product yields
\[
\log F_g(x)
=
\sum_{m\ge1}\beta_g(m)\log(1-x^m).
\]
Using the Maclaurin series expansion \(\log(1-x^m) = -\sum_{r\ge1}\frac{x^{mr}}{r}\), we obtain
\[
\log F_g(x)
=
-\sum_{m,r\ge1}\frac{\beta_g(m)}r x^{mr}.
\]
Collecting the coefficients of \(x^n\) gives
\[
\log F_g(x)
=
-\sum_{n\ge1}
\left(
\frac1n
\sum_{d\mid n} d\,\beta_g(d)
\right)x^n.
\]
Comparing this directly with the definition \(\log F_g(x) = -\sum_{n\ge1}\frac{g(n)}n x^n\), we find
\[
g(n)=\sum_{d\mid n}d\,\beta_g(d).
\]
Applying classical Möbius inversion yields the desired formula for \(\beta_g(n)\).
\end{proof}

This structural identity immediately allows us to express the coefficients of the series expansion in terms of partial Bell polynomials.

\begin{corollary}[Bell-polynomial representation]
The Bell transform \(F_g(x)\) admits a power-series expansion
\[
F_g(x)
=
\sum_{n\ge0}a_g(n)x^n,
\qquad a_g(0)=1,
\]
where the coefficients \(a_g(n)\) admit the explicit closed form
\[
a_g(n)
=
\frac1{n!}
B_n\!\bigl(
-0!g(1),\,
-1!g(2),\,
\dots,\,
-(n-1)!g(n)
\bigr),
\]
and \(B_n\) denotes the complete exponential Bell polynomial.
\end{corollary}

\begin{proof}
Using the classical generating identity for complete Bell polynomials,
\[
\exp\!\left(
\sum_{m\ge1}x_m\frac{t^m}{m!}
\right)
=
\sum_{n\ge0}\frac{B_n(x_1,\dots,x_n)}{n!}t^n,
\]
we make the substitution \(x_m=-(m-1)!g(m)\). Substituting this into the exponential definition of \(F_g(x)\) directly yields the stated formula.
\end{proof}

\begin{corollary}[Power-series recurrence expansion]
The coefficients \(a_g(n)\) of the power-series expansion \(F_g(x) = \sum_{n\ge0}a_g(n)x^n\) satisfy the linear recurrence relation
\[
n\,a_g(n)
=
-\sum_{k=1}^n g(k)\,a_g(n-k),
\qquad (n\ge1).
\]
\end{corollary}

\begin{proof}
By the preceding corollary, we have \( a_g(n) = \frac{1}{n!} B_n(x_1,\dots,x_n) \), where \(x_m=-(m-1)!g(m)\). The complete Bell polynomials satisfy the recursive identity\
\[
B_{n+1}(x_1,\dots,x_{n+1})
=
\sum_{i=0}^{n}
\binom{n}{i}
B_{n-i}(x_1,\dots,x_{n-i})\,x_{i+1},
\]
with \(B_0=1\). Substituting \(x_{i+1}=-i!g(i+1)\) gives
\[
B_{n+1}
=
-\sum_{i=0}^{n}
\binom{n}{i}
B_{n-i}\, i!\, g(i+1).
\]
Dividing both sides by \((n+1)!\) and utilizing the identity \(\binom{n}{i}i! = \frac{n!}{(n-i)!}\) yields
\[
a_g(n+1)
=
-\frac1{n+1}
\sum_{i=0}^{n}
\frac{B_{n-i}}{(n-i)!}\,
g(i+1).
\]
Recognizing that \(a_g(n-i)=\frac{B_{n-i}}{(n-i)!}\), we find
\[
(n+1)a_g(n+1)
=
-\sum_{i=0}^{n}
g(i+1)\,a_g(n-i).
\]
Shifting the summation index by setting \(k = i+1\) produces the desired recurrence relation for \(n\,a_g(n)\).
\end{proof}

\begin{theorem}[Congruence Inheritance via the Bell Transform] \label{thm:congruence}
Let \(p\) be a prime, \(g(n)\) be an arithmetic function, and let \(\beta_g(m) = \frac1m(\mu*g)(m)\) be its associated Bell exponents. If
\[
\beta_g(m)\equiv0\pmod p
\qquad
\text{for all } m\in(\mathbb Z/p\mathbb Z)^\times.
\]
then the Bell-transform coefficients \(F_g(x) = \sum_{n\ge0}a_g(n)x^n\) satisfy the congruence
\[
\boxed{
a_g(n)\equiv0\pmod p
\qquad \text{for all } n\in(\mathbb Z/p\mathbb Z)^\times.
}
\]
\end{theorem}

\begin{proof}
We proceed by strong induction on \(n\) using the Bell recurrence relation. First, observe that the inverse Bell formula gives \(g(n) = \sum_{d\mid n} d\,\beta_g(d)\).

Assume that \(p\nmid n\). Then every divisor \(d\mid n\) also satisfies \(p\nmid d\). By hypothesis, \(\beta_g(d)\equiv0\pmod p\) for every such divisor. Hence,
\begin{equation}
g(n)
=
\sum_{d\mid n} d\,\beta_g(d)
\equiv0\pmod p \qquad (p\nmid n).
\label{eq:gn_mod_p}
\end{equation}

We now prove \(a_g(n)\equiv0\pmod p\) for \(p\nmid n\) by induction. The base case \(a_g(0)=1\) is unaffected as \(p\nmid0\) does not apply. Assume inductively that \(a_g(m)\equiv0\pmod p\) for all \(m<n\) where \(p\nmid m\). 

For an integer \(n\) such that \(p\nmid n\), we analyze the recurrence relation \(n\,a_g(n) = -\sum_{k=1}^n g(k)a_g(n-k)\) modulo \(p\). For any term \(g(k)a_g(n-k)\):
\begin{itemize}
    \item \textbf{Case 1 (\(p\nmid k\)):} By \eqref{eq:gn_mod_p}, \(g(k)\equiv0\pmod p\), nullifying the term.
    \item \textbf{Case 2 (\(p\mid k\)):} Since \(p\nmid n\), it follows that \(p\nmid(n-k)\). Because \(n-k < n\), the inductive hypothesis guarantees \(a_g(n-k)\equiv0\pmod p\), nullifying the term.
\end{itemize}
Thus, every term in the recurrence vanishes modulo \(p\), leaving \(n\,a_g(n)\equiv0\pmod p\). Because \(p\nmid n\), \(n\) is invertible modulo \(p\), forcing \(a_g(n)\equiv0\pmod p\).
\end{proof}

\begin{corollary}[Congruences for Ramanujan's Tau Function]
Let the modular discriminant be defined as
\[
\Delta(q)
=
q\prod_{m\ge1}(1-q^m)^{24}
=
\sum_{n\ge1}\tau(n)q^n.
\]
Then \(\tau(2n)\equiv0\pmod2\), and \(\tau(3n-1)\equiv0\pmod3\), \(\tau(3n)\equiv0\pmod3\).
\end{corollary}

\begin{proof}
Extracting the power of \(q\), the Bell exponents of \(\Delta(q)/q\) are constant: \(\beta(m)=24\) for all \(m\ge1\). Since \(24\equiv0\pmod2\), Theorem \ref{thm:congruence} with \(p=2\) yields \(\tau(n+1)\equiv0\pmod2\) for \(2\nmid n\). Writing \(n=2m-1\), we obtain \(\tau(2m)\equiv0\pmod2\). 

Similarly, \(24\equiv0\pmod3\). Applying Theorem \ref{thm:congruence} with \(p=3\) gives \(\tau(n+1)\equiv0\pmod3\) for \(3\nmid n\). This implies \(n\equiv1,2\pmod3\), which yields the congruences for \(\tau(3m-1)\) and \(\tau(3m)\).
\end{proof}

\begin{corollary}[Prime-Colored Partition Congruences]
Let \(p_k(n)\) denote the number of \(k\)-colored partitions of \(n\), generated by \(P_k(x) = \prod_{m\ge1}(1-x^m)^{-k}\). If \(p\) is prime, then \(p_p(n)\equiv0\pmod p\) whenever \(p\nmid n\).
\end{corollary}

\begin{proof}
The Bell exponents are \(\beta(m)=-p\), which trivially satisfies \(\beta(m)\equiv0\pmod p\) for all \(m\). Applying Theorem \ref{thm:congruence} yields the result.
\end{proof}

The logic of Theorem \ref{thm:congruence} can be upgraded from congruences to exact equality if the exponents vanish entirely.

\begin{theorem}[Exact Vanishing via the Bell Transform] \label{thm:vanishing}
Let \(p\) be a prime and let \(g(n)\) be an arithmetic function. Suppose that the Bell exponents satisfy \(\beta_g(m)=0\) whenever \(p\nmid m\). Then the Bell-transform coefficients satisfy the exact vanishing property
\[
\boxed{
a_g(n)=0
\qquad
\text{whenever }p\nmid n.
}
\]
\end{theorem}

\begin{proof}
The proof proceeds by strong induction identically to Theorem \ref{thm:congruence}, replacing congruence modulo \(p\) with strict equality to zero.
\end{proof}

\begin{corollary}[Vanishing Coefficients of Cyclotomic Polynomials]
Let \(p\) be a prime such that \(p^2\mid n\). Let \(\Phi_n(x) = \sum_{k=0}^{\varphi(n)} a_n(k)x^k\) denote the \(n\)-th cyclotomic polynomial. Then \(a_n(k)=0\) whenever \(p\nmid k\).
\end{corollary}

\begin{proof}
The cyclotomic polynomial admits the finite product representation \(\Phi_n(x) = \prod_{m\mid n} (1-x^m)^{\mu(n/m)}\). Thus, the Bell exponents are \(\beta(m) = \mu(n/m)\) if \(m\mid n\), and \(0\) otherwise. 

Assume \(p\nmid m\). If \(m\mid n\), since \(p^2\mid n\), the quotient \(n/m\) must retain the factor of \(p^2\). Because the Möbius function vanishes on integers divisible by a prime square, \(\mu(n/m)=0\). Hence \(\beta(m)=0\) for all \(m\) coprime to \(p\). Theorem \ref{thm:vanishing} directly implies \(a_n(k)=0\) for \(p\nmid k\).
\end{proof}
\section{Applications to Classical Arithmetic Functions}

The Bell-transform framework becomes especially transparent when applied to structured arithmetic functions. 
For each arithmetic driver \(g(n)\), we state:

\begin{enumerate}
    \item the generating function \(F_g(x)\),
    \item the corresponding Euler product,
    \item the Bell recurrence relation,
    \item and the first few coefficients of the associated power series.
\end{enumerate}

\subsection{Completely Multiplicative Functions}

\subsubsection{The Dirichlet Identity Function}

Let \(\varepsilon(n)\) denote the Dirichlet identity:
\[
\varepsilon(1)=1,
\qquad
\varepsilon(n)=0
\quad(n\ge2).
\]

The Bell transform becomes
\[
F_\varepsilon(x)
=
e^{-x}
=
\prod_{n\ge1}
(1-x^n)^{\mu(n)/n}.
\]

\subsubsection{Power Functions}

For the arithmetic driver
\[
g(n)=n^k,
\]
the Bell exponents become
\[
\beta(n)=\frac{J_k(n)}{n},
\]
where \(J_k\) denotes the Jordan totient function.

The generating function is
\[
F_{I_k}(x)
=
e^{-\Li_{1-k}(x)}
=
\prod_{n\ge1}
(1-x^n)^{J_k(n)/n}=\sum_{n=0}^{\infty} a_{I_k}(n) x^n
\]

The coefficients satisfy the recurrence
\[
\boxed{
n\,a_{I_k}
=
-\sum_{j=1}^n
j^k\,a_{I_k}(n-j).
}
\]

The first terms are
\[
\begin{aligned}
a_{I_k}(0)&=1,\\
a_{I_k}(1)&=-1,\\
a_{I_k}(2)&=\frac{1-2^k}{2},\\
a_{I_k}(3)&=\frac{-2+3\cdot2^k-2\cdot3^k}{6},\\
a_{I_k}(4)&=
\frac{
6-11\cdot2^k+6\cdot2^{2k}
+8\cdot3^k-6\cdot4^k
}{24}.
\end{aligned}
\]

\subsubsection{Dirichlet Characters}

Let \(\chi_4(n)\) denote the quadratic character modulo \(4\):
\[
\chi_4(n)
=
\begin{cases}
0,&2\mid n,\\
1,&n\equiv1\pmod4,\\
-1,&n\equiv3\pmod4.
\end{cases}
\]

The Bell transform gives
\[
F_{\chi_4}(x)
=
e^{-\arctan x}
=
\prod_{n\ge1}
(1-x^n)^{
(\mu*\chi_4)(n)/n
}=\sum_{n=0}^{\infty} a_{\chi_4}(n) x^n
\]

The recurrence becomes
\[
\boxed{
n\,a_{\chi_4}(n)
=
-\sum_{j=1}^n
\chi_4(j)\,a_{\chi_4}(n-j).
}
\]

The first coefficients are
\[
\begin{aligned}
a_{\chi_4}(0)&=1,\\
a_{\chi_4}(1)&=-1,\\
a_{\chi_4}(2)&=\frac12,\\
a_{\chi_4}(3)&=-\frac16,\\
a_{\chi_4}(4)&=-\frac5{24}.
\end{aligned}
\]

Thus
\[
e^{-\arctan x}
=
1-x+\frac{x^2}{2}-\frac{x^3}{6}
-\frac5{24}x^4+\cdots.
\]

\subsection{Multiplicative Functions}

\subsubsection{Euler's Totient Function}

For the arithmetic driver
\[
g(n)=\varphi(n),
\]

\[
F_\varphi(x)
=
\exp\!\left(
\frac{x}{x-1}
\right)
=
\prod_{n\ge1}
(1-x^n)^{
(\mu*\varphi)(n)/n
}=\sum_{n=0}^{\infty}a_\varphi(n)x^n\]

The coefficients satisfy
\[
\boxed{
n\,a_\varphi(n)
=
-\sum_{j=1}^n
\varphi(j)\,a_\varphi(n-j).
}
\]

Expanding recursively:
\[
\begin{aligned}
a_\varphi(0)&=1,\\
a_\varphi(1)&=-1,\\
a_\varphi(2)&=-\frac12,\\
a_\varphi(3)&=-\frac16,\\
a_\varphi(4)&=\frac1{24}.
\end{aligned}
\]

Hence
\[
\exp\!\left(
\frac{x}{x-1}
\right)
=
1-x-\frac{x^2}{2}
-\frac{x^3}{6}
+\frac{x^4}{24}
+\cdots.
\]

\subsubsection{Ramanujan Sums and Cyclotomic Polynomials}

For the arithmetic driver
\[
g(n)=c_q(n),
\]
the Bell transform produces the cyclotomic polynomial:
\[
F_{c_q}(x)
=
\Phi_q(x)
=
\prod_{d\mid q}
(1-x^d)^{\mu(q/d)}=\sum_{n=0}^{\infty}a_{c_q}(n)x^n
\]

The coefficients satisfy
\[
\boxed{
n\,a_{c_q}(n)
=
-\sum_{j=1}^n
c_q(j)\,(n-j).
}
\]

\subsection{Completely Additive Functions}

\subsubsection{The Natural Logarithm}

For the arithmetic driver
\[
g(n)=\log n,
\]
the identity
\[
\mu*\log=\Lambda
\]
gives the Euler product
\[
F_{\log}(x)
=
\exp\!\left(
-\sum_{n\ge1}\frac{\log n}{n}x^n
\right)
=
\prod_{n=1}^{\infty}
(1-x^n)^{\Lambda(n)/n}=\sum_{n=0}^{\infty}a_{log}(n)x^n
\]

The recurrence becomes
\[
\boxed{
n\,a_{log}(n)
=
-\sum_{j=1}^n
(\log j)\,a_{log}(n-j).
}
\]

The first coefficients are
\[
\begin{aligned}
a_{log}(0)&=1,\\
a_{log}(1)&=0,\\
a_{log}(2)&=-\frac{\log2}{2},\\
a_{log}(3)&=-\frac{\log3}{3},\\
a_{log}(4)&=
\frac{\log^2(2)-2\log4}{8}.
\end{aligned}
\]

\subsection{Sums of Squares Functions}

\subsubsection{Sums of Four Squares}

Jacobi's four-square theorem gives
\[
r_4(n)=8\sigma(n)-32\sigma(n/4).
\]

The Bell transform becomes
\[
F_{r_4}(x)
=\prod_{n=1}^{\infty}\left(\frac{1-x^n}{1-x^{4n}}\right)^8=\sum_{n=0}^{\infty}a_{r_4}(n)x^n
\]

The recurrence relation is
\[
\boxed{
n\,a_{r_4}n
=
-\sum_{j=1}^n
r_4(j)\,a_{r_4}(n-j).
}
\]

The first coefficients are
\[
\begin{aligned}
a_{r_4}(0)&=1,\\
a_{r_4}(1)&=-8,\\
a_{r_4}(2)&=24,\\
a_{r_4}(3)&=-32,\\
a_{r_4}(4)&=24.
\end{aligned}
\]
and also using Theorem \ref{thm:congruence}

\[
\boxed{
a_{r_4}(n)\equiv0\pmod2
\qquad(n\ge1).
}
\]
\section{Applications to Special Polynomial Families}

The Bell-transform framework also acts as a discrete engine for classical polynomial sequences. 
Each generating function yields:
\begin{itemize}
    \item an arithmetic driver \(g(n)\),
    \item a Bell recurrence,
    \item and an explicit coefficient expansion.
\end{itemize}

\subsection{Appell Sequences}

\subsubsection{Bernoulli Polynomials \(B_n(x)\)}

The Bernoulli polynomials are generated by
\[
\frac{te^{xt}}{e^t-1}
=
\sum_{n=0}^{\infty}
B_n(x)\frac{t^n}{n!}.
\]

Comparison with the Bell-transform form gives the arithmetic driver
\[
g(n)
=
\frac{B_n}{(n-1)!}
-
x\delta_{n,1}.
\]

The Bell recurrence becomes
\[
\boxed{
B_n(x)
=
-\sum_{k=1}^n
\binom{n-1}{k-1}
B_k\,B_{n-k}(x)
+
xn\,B_{n-1}(x).
}
\]

Expanding recursively gives
\[
\begin{aligned}
B_0(x)&=1,\\
B_1(x)&=x-\frac12,\\
B_2(x)&=x^2-x+\frac16,\\
B_3(x)&=x^3-\frac32x^2+\frac12x,\\
B_4(x)&=x^4-2x^3+x^2-\frac1{30}.
\end{aligned}
\]

\subsubsection{Euler Polynomials \(E_n(x)\)}

The Euler polynomials are generated by
\[
\frac{2e^{xt}}{e^t+1}
=
\sum_{n=0}^{\infty}
E_n(x)\frac{t^n}{n!}.
\]

The Bell-transform arithmetic driver is
\[
g(n)
=
\frac{(2^n-1)B_n}{(n-1)!}
+
(1-x)\delta_{n,1}.
\]

The recurrence relation is
\[
\boxed{
E_n(x)
=
-\sum_{k=1}^n
\binom{n-1}{k-1}
(2^k-1)B_k\,E_{n-k}(x)
+
(x-1)nE_{n-1}(x).
}
\]

The first few polynomials are
\[
\begin{aligned}
E_0(x)&=1,\\
E_1(x)&=x-\frac12,\\
E_2(x)&=x^2-x,\\
E_3(x)&=x^3-\frac32x^2+\frac14,\\
E_4(x)&=x^4-2x^3+x.
\end{aligned}
\]

\subsubsection{Hermite Polynomials \(H_n(x)\)}

The physicists' Hermite polynomials satisfy
\[
\exp(2xt-t^2)
=
\sum_{n=0}^{\infty}
H_n(x)\frac{t^n}{n!}.
\]

The Bell arithmetic driver is
\[
g(1)=-2x,
\qquad
g(2)=2,
\qquad
g(n)=0
\quad(n\ge3).
\]

The Bell recurrence collapses to the classical three-term relation:
\[
\boxed{
H_n(x)
=
2xH_{n-1}(x)
-
2(n-1)H_{n-2}(x).
}
\]

Expanding recursively:
\[
\begin{aligned}
H_0(x)&=1,\\
H_1(x)&=2x,\\
H_2(x)&=4x^2-2,\\
H_3(x)&=8x^3-12x,\\
H_4(x)&=16x^4-48x^2+12.
\end{aligned}
\]

\subsection{Sheffer Sequences}

\subsubsection{Touchard Polynomials \(T_n(x)\)}

The Touchard polynomials are generated by
\[
\exp(x(e^t-1))
=
\sum_{n=0}^{\infty}
T_n(x)\frac{t^n}{n!}.
\]

The Bell arithmetic driver is
\[
g(n)
=
-\frac{x}{(n-1)!}.
\]

The recurrence becomes
\[
\boxed{
T_n(x)
=
x\sum_{j=0}^{n-1}
\binom{n-1}{j}
T_j(x).
}
\]

The first few polynomials are
\[
\begin{aligned}
T_0(x)&=1,\\
T_1(x)&=x,\\
T_2(x)&=x+x^2,\\
T_3(x)&=x+3x^2+x^3,\\
T_4(x)&=x+7x^2+6x^3+x^4.
\end{aligned}
\]

\subsubsection{Generalized Laguerre Polynomials \(L_n^{(\alpha)}(x)\)}

The generating function is
\[
(1-t)^{-\alpha-1}
\exp\!\left(
-\frac{xt}{1-t}
\right)
=
\sum_{n=0}^{\infty}
L_n^{(\alpha)}(x)t^n.
\]

The Bell arithmetic driver is
\[
g(n)=xn-(\alpha+1).
\]

The recurrence relation becomes
\[
\boxed{
nL_n^{(\alpha)}(x)
=
-\sum_{k=1}^n
(xk-\alpha-1)L_{n-k}^{(\alpha)}(x).
}
\]

The first few polynomials are
\[
\begin{aligned}
L_0^{(\alpha)}(x)&=1,\\
L_1^{(\alpha)}(x)&=-x+\alpha+1,\\
L_2^{(\alpha)}(x)&=\frac12
\left(
x^2-2(\alpha+2)x+(\alpha+1)(\alpha+2)
\right),\\
L_3^{(\alpha)}(x)&=
-\frac{x^3}{6}
+\frac{\alpha+3}{2}x^2
-\frac{(\alpha+2)(\alpha+3)}2x
+\frac{(\alpha+1)(\alpha+2)(\alpha+3)}6.
\end{aligned}
\]

\subsubsection{Poisson--Charlier Polynomials \(C_n(x;a)\)}

The generating function is
\[
e^{-at}
\left(
1+\frac{t}{a}
\right)^x
=
\sum_{n=0}^{\infty}
C_n(x;a)\frac{t^n}{n!}.
\]

The Bell arithmetic driver is
\[
g(1)=a-\frac xa,
\qquad
g(n)=\frac{x(-1)^n}{a^n}
\quad(n\ge2).
\]

The recurrence becomes
\[
\boxed{
C_n(x;a)
=
-\left(a-\frac xa\right)C_{n-1}(x;a)
-
\sum_{k=2}^n
\binom{n-1}{k-1}
\frac{x(-1)^k}{a^k}
C_{n-k}(x;a).
}
\]

The first few polynomials are
\[
\begin{aligned}
C_0(x;a)&=1,\\
C_1(x;a)&=\frac{x}{a}-a,\\
C_2(x;a)&=
\frac{x^2-x}{a^2}
-2x+a^2,\\
C_3(x;a)&=
\frac{x^3-3x^2+2x}{a^3}
-\frac{3(x^2-x)}a
+3ax-a^3.
\end{aligned}
\]

\end{document}